\documentclass[reqno]{amsart}
\usepackage{amssymb,mathtools,graphicx}%
\usepackage[hmarginratio=1:1]{geometry}
\usepackage{ifpdf}
\ifpdf
\usepackage[hyperindex]{hyperref}
\else
\expandafter\ifx\csname dvipdfm\endcsname\relax
\usepackage[hypertex,hyperindex]{hyperref}
\else
\usepackage[dvipdfm,hyperindex]{hyperref}
\fi
\fi

\theoremstyle{plain}
\newtheorem{thm}{Theorem}[section]
\newtheorem{cor}{Corollary}[section]
\theoremstyle{remark}
\newtheorem{rem}{Remark}[section]
\numberwithin{equation}{section}
\allowdisplaybreaks[4]
\newcommand{\td}{\textup{d}}
\DeclareMathOperator{\bell}{B}

\begin{document}

\title[An analytic generalization of the Catalan numbers]
{An analytic generalization of the Catalan numbers and its integral representation}

\author[W.-H. Li]{Wen-Hui Li}
\address[Li]{Department of Fundamental Courses, Zhengzhou University of Science and Technology, Zhengzhou 450064, Henan, China}
\email{\href{mailto: W.-H. Li <wen.hui.li102@gmail.com>}{wen.hui.li102@gmail.com}, \href{mailto: W.-H. Li <wen.hui.li@foxmail.com>}{wen.hui.li@foxmail.com}}
\urladdr{\url{https://orcid.org/0000-0002-1848-8855}}

\author[J. Cao]{Jian Cao}
\address[Cao]{Department of Mathematics, Hangzhou Normal University, Hangzhou 311121, Zhejiang, China}
\email{\href{mailto: J. Cao <21caojian@gmail.com>}{21caojian@gmail.com}, \href{mailto: J.Cao <21caojian@163.com>}{21caojian@163.com}}
\urladdr{\url{https://orcid.org/0000-0002-7173-0591}}

\author[D.-W. Niu]{Da-Wei Niu}
\address[Niu]{Department of Science, Henan University of Animal Husbandry and Economy, Zhengzhou 450046, Henan, China}
\email{\href{mailto: D.-W. Niu <nnddww@gmail.com>}{nnddww@gmail.com}, \href{mailto: D.-W. Niu <nnddww@163.com>}{nnddww@163.com}}
\urladdr{\url{https://orcid.org/0000-0003-4033-7911}}

\author[J.-L. Zhao]{Jiao-Lian Zhao}
\address[Zhao]{School of Mathematics and Statistics, Weinan Normal University, Weinan 714000, Shaanxi, China}
\email{\href{mailto: J.-L. Zhao <zhaojl2004@gmail.com>}{zhaojl2004@gmail.com}, \href{mailto: J.-L. Zhao <zhaojl@wnu.edu.cn>}{zhaojl@wnu.edu.cn}}
\urladdr{\url{https://orcid.org/0000-0003-3173-7075}}

\author[F. Qi]{Feng Qi*}
\address[Qi]{Institute of Mathematics, Henan Polytechnic University, Jiaozuo 454010, Henan, China; School of Mathematical Sciences, Tiangong University, Tianjin 300387, China}
\email{\href{mailto: F. Qi <qifeng618@gmail.com>}{qifeng618@gmail.com}, \href{mailto: F. Qi <qifeng618@hotmail.com>}{qifeng618@hotmail.com}, \href{mailto: F. Qi <qifeng618@qq.com>}{qifeng618@qq.com}}
\urladdr{\url{https://qifeng618.wordpress.com}, \url{https://orcid.org/0000-0001-6239-2968}}

\dedicatory{Dedicated to people facing and battling COVID-19}

\begin{abstract}
In the paper, the authors analytically generalize the Catalan numbers in combinatorial number theory, establish an integral representation of the analytic generalization of the Catalan numbers by virtue of Cauchy's integral formula in the theory of complex functions, and point out potential directions to further study.
\end{abstract}



\thanks{*Corresponding author}

\thanks{This paper was typeset using \AmS-\LaTeX}

\maketitle
\tableofcontents

\section{Backgrounds and motivations}

The Catalan numbers
\begin{equation}\label{Catalan-1Exp}
C_n=\frac{1}{n+1}\binom{2n}{n}=\frac{4^n\Gamma(n+1/2)}{\sqrt\pi\,\Gamma(n+2)}
\end{equation}
form a sequence of integers~\cite{Grimaldi-B2012, Koshy-B-2009, Spivey-art-2019}, have combinatorial interpretations~\cite{Koshy-B-2009, Roman-Catalan-B2015}, have a long history~\cite{Grimaldi-B2012, Catalan-Int-Surv.tex}, and can be generated~\cite{Roman-Catalan-B2015, Stanley-Catalan-2015} by
\begin{equation}\label{CatalanNo-Gen=F}
G(x)=\frac2{1+\sqrt{1-4x}\,}=\sum_{n=0}^\infty C_nx^n,
\end{equation}
where
\begin{equation*}
\Gamma(z)=\lim_{n\to\infty}\frac{n!n^z}{\prod_{k=0}^n(z+k)}, \quad z\ne0,-1,-2,\dotsc
\end{equation*}
is the classical Euler gamma function~\cite{NIST-HB-2010, Ext-Alzer-Ouimet-Qi.tex}.
\par
In~the electronic preprint~\cite{LCMF-Catalan-NS.tex} and its formally published version~\cite{195-2017-JOCAAA.tex}, starting from the second expression in terms of gamma functions in~\eqref{Catalan-1Exp}, the Catalan numbers $C_n$ were analytically generalized to generalized Catalan function
\begin{equation}\label{C(a,b;x)=dfn}
C(a,b;z)=\frac{\Gamma(b)}{\Gamma(a)} \biggl(\frac{b}a\biggr)^z \frac{\Gamma(z+a)}{\Gamma(z+b)}, \quad \Re(a),\Re(b)>0, \quad \Re(z)\ge0
\end{equation}
with
\begin{equation}\label{Cataln-Qi-Catalan-Rel}
C\biggl(\frac12,2;n\biggr)=C_n, \quad n\ge0.
\end{equation}
Hereafter, generalized Catalan function $C(a,b;z)$ and its analytic generalizations were deeply investigated in~\cite{Chammam-India-Catalan-Qi, LCMF-Catalan.tex, Mathematics-129120.tex, Dana-Picard-IJMEST-2012.tex, DANA-PICARD.tex, Ars-Comb-Catalan-Asymp-Qi.tex, Akkurt-Qi-Catalan-JCAA.tex, Qi-Fuss-Catalan-One.tex, Catalan-Int-Surv.tex, Qi-Guo-Sapientiae-2016.tex, Catalan-Qi-Funct-T.tex, MAAST-Qi.tex, MTJPAM-D-19-00007.tex, Catalan-GF-Plus.tex, Fuss-Catalan-Qi-S.tex, Catalan-Number-S.tex, Schur-Catalan-Gen.tex, Kims-Rus-Catalan.tex, K2-S-D-JCAA17.tex, GJMA-5055-Catalan-Shi-Liu-Qi.tex, Catalan-Qi.tex, applications.tex} and closely related references therein.
\par
The Catalan numbers $C_n$ for $n\ge0$ have several integral representations which have been surveyed in~\cite[Section~2]{Catalan-Int-Surv.tex}.
The integral representation
\begin{equation}\label{Catalan-No-Int-Wiki}
C_n=\frac1{2\pi}\int_0^4\sqrt{\frac{4-x}x}\,x^n\td x, \quad n\ge0
\end{equation}
was discovered in~\cite{Article-01.2.5} and applied in~\cite{K2-S-D-JCAA17.tex}. An alternative integral representation
\begin{equation}\label{Catalan-No-Int-Qi}
C_n=\frac1{\pi}\int_0^\infty\frac{\sqrt{t}\,}{(t+1/4)^{n+2}}\td t
\end{equation}
was derived from the integral representation
\begin{equation}\label{real-integral-form}
\frac1{1+\sqrt{1-4x}\,}=\frac1{2\pi}\int_0^{\infty}\frac{\sqrt{t}\,}{1/4+t}\frac{1}{1/4+t-x}\td t, \quad x\in\biggl(-\infty,\frac14\biggr]
\end{equation}
which was established in~\cite[Theorem~1.3]{Catalan-Number-S.tex} by virtue of Cauchy's integral formula in the theory of complex functions.
\par
The generalized Catalan function $C(a,b;z)$ defined by~\eqref{C(a,b;x)=dfn} has also several integral representations which have been surveyed in~\cite[Section~2]{Catalan-Int-Surv.tex}. For example, corresponding to integral representations in~\eqref{Catalan-No-Int-Wiki} and~\eqref{Catalan-No-Int-Qi}, integral representations
\begin{equation}\label{Gen-Catalan-Int-Short-Eq}
C(a,b;x)=\biggl(\frac{a}b\biggr)^{b-1} \frac1{B(a,b-a)} \int_0^{b/a}\biggl(\frac{b}a-t\biggr)^{b-a-1}t^{x+a-1}\td t
\end{equation}
and
\begin{equation}\label{Gen-Catalan-Int-Eq}
C(a,b;x)=\biggl(\frac{a}{b}\biggr)^{a}\frac1{B(a,b-a)}\int_0^\infty\frac{t^{b-a-1}}{(t+a/b)^{x+b}}\td t.
\end{equation}
for $b>a>0$ and $x\ge0$ were established in~\cite[Theorem~4]{Catalan-GF-Plus.tex}, where the classical beta function $B(z,w)$ can be defined or expressed~\cite{DANA-PICARD.tex, Temme-96-book} by
\begin{equation*}
B(z,w)=\int_0^1t^{z-1}(1-t)^{w-1}\td t
=\int_0^\infty\frac{t^{z-1}}{(1+t)^{z+w}}\td t
=\frac{\Gamma(z)\Gamma(w)}{\Gamma(z+w)}
\end{equation*}
for $\Re(z),\Re(w)>0$. We note that, when letting $a=\frac12$ and $b=2$, the integral representations~\eqref{Gen-Catalan-Int-Short-Eq} and~\eqref{Gen-Catalan-Int-Eq} become those in~\eqref{Catalan-No-Int-Wiki} and~\eqref{Catalan-No-Int-Qi} respectively.
\par
The generating function $G(x)$ in~\eqref{CatalanNo-Gen=F} can be regarded as a special case $a=\frac12$, $b=\frac14$, and $c=1$ of the function
\[
G_{a,b,c}(x)=\frac1{a+\sqrt{b-cx}\,}, \quad a\ge0,b,c>0.
\]
Essentially, it is better to regard the function
\begin{equation}\label{G(a-b-x)-GF}
G_{a,b}(x)=\frac1{a+\sqrt{b-x}\,}, \quad a\ge0,b>0
\end{equation}
as a generalization of the generating function $G(x)$, because
\begin{equation*}
G_{1/2,1/4}(x)=G(x), \quad G_{a,b}(x)=G_{a,b,1}(x), \quad G_{a,b,c}(x)=\frac{G_{a/\sqrt{c}\,,b/c}(x)}{\sqrt{c}\,},
\end{equation*}
but we can not express $G_{a,b}(x)$ in terms of $G(x)$.
\par
Now we would like to pose the following three problems.
\begin{enumerate}
\item
Can one establish an explicit formula for the sequence $\mathcal{C}_n(a,b)$ generated by
\begin{equation}\label{C-n(a-b)-gen-funct}
G_{a,b}(x)=\frac1{a+\sqrt{b-x}\,}=\sum_{k=0}^{\infty}\mathcal{C}_n(a,b)x^n
\end{equation}
for $a\ge0$ and $b>0$?
\item
Can one find an integral representation for the sequence $\mathcal{C}_n(a,b)$ by finding an integral representation of the generating function $G_{a,b}(x)$ in~\eqref{G(a-b-x)-GF}?
\item
Can one combinatorially interpret the sequence $\mathcal{C}_n(a,b)$ or some special case of $\mathcal{C}_n(a,b)$ except the case $a=\frac12$ and $b=\frac14$?
\end{enumerate}
\par
It is easy to see that
\begin{equation}\label{Catlan-Qi-a2zero}
\lim_{a\to0^+}\mathcal{C}_n(a,b)=\frac{(-1)^n}{n!}\biggl\langle-\frac12\biggr\rangle_n\frac{1}{b^{(2n+1)/2}}
\end{equation}
and
\begin{equation}\label{C-n(ab)=C-n-Eq}
\mathcal{C}_n\biggl(\frac12,\frac14\biggr)=C_n
\end{equation}
for $n\ge0$, where the notation
\begin{equation*}
\langle\alpha\rangle_n=
\prod_{k=0}^{n-1}(\alpha-k)=
\begin{cases}
\alpha(\alpha-1)\dotsm(\alpha-n+1), & n\ge1\\
1,& n=0
\end{cases}
\end{equation*}
for $\alpha\ne0$ is called the falling factorial~\cite{CDM-68111.tex, Bell-value-elem-funct.tex, Note-Elem-UPB-Qi-Zhang-Li.tex}.
Comparing~\eqref{C-n(ab)=C-n-Eq} with~\eqref{Cataln-Qi-Catalan-Rel} reveals that $C(a,b;n)\not\equiv\mathcal{C}_n(a,b)$, although it is possible that
\begin{equation*}
\{C(a,b;n): n\ge0,a\ge0,b>0\}=\{\mathcal{C}_n(a,b): n\ge0,a\ge0,b>0\}
\end{equation*}
or that there exist two $2$-tuples $(a_n,b_n)\in(0,\infty)\times(0,\infty)$ and $(\alpha_n,\beta_n)\in(0,\infty)\times(0,\infty)$ such that $C(a_n,b_n;n)=\mathcal{C}_n(\alpha_n,\beta_n)$ for all $n\ge0$.
\par
For our own convenience and referencing to the convention in mathematical community, while calling $C(a,b;n)$ for $n\ge0$, $a\ge0$, and $b>0$ generalized Catalan numbers of the first kind, we call $\mathcal{C}_n(a,b)$ for $n\ge0$, $a\ge0$, and $b>0$ generalized Catalan numbers of the second kind.
\par
In this paper, we will give solutions to the first two problems above: establishing an explicit formula for generalized Catalan numbers of the second kind $\mathcal{C}_n(a,b)$ and finding an integral representation for generalized Catalan numbers of the second kind $\mathcal{C}_n(a,b)$ by finding an integral representation of the generating function $G_{a,b}(x)$ in~\eqref{G(a-b-x)-GF}, while leaving the third problem above to interested combinatorists.

\section{An explicit formula for generalized Catalan numbers of the second kind}

In this section, we will establish an explicit formula for generalized Catalan numbers of the second kind $\mathcal{C}_n(a,b)$, which gives a solution to the first problem posed on page~\pageref{C-n(a-b)-gen-funct}.

\begin{thm}\label{C-n(a-b)-explicit-thm}
The generalized Catalan numbers of the second kind $\mathcal{C}_n(a,b)$ for $n\ge0$, $a\ge0$, and $b>0$ can be explicitly computed by
\begin{equation}\label{C-n(a-b)-explicit}
\mathcal{C}_n(a,b)=\frac{1}{(2n)!!b^{n+1/2}}\sum_{k=0}^{n}\binom{2n-k-1}{2(n-k)}\frac{k![2(n-k)-1]!!}{\bigl(1+a/\sqrt{b}\,\bigr)^{k+1}},
\end{equation}
where the double factorial of negative odd integers $-(2\ell+1)$ is defined by
\begin{equation*}
(-2\ell-1)!!=\frac{(-1)^\ell}{(2\ell-1)!!}=(-1)^\ell\frac{(2\ell)!!}{(2\ell)!}, \quad \ell\ge0.
\end{equation*}
\end{thm}

\begin{proof}
The Bell polynomials of the second kind $\bell_{n,k}(x_1,x_2,\dotsc,x_{n-k+1})$ for $n\ge k\ge0$ are defined in~\cite[p.~134, Theorem~A]{Comtet-Combinatorics-74} by
\begin{equation*}
\bell_{n,k}(x_1,x_2,\dotsc,x_{n-k+1})=\sum_{\substack{1\le i\le n-k+1\\ \ell_i\in\{0\}\cup\mathbb{N}\\ \sum_{i=1}^{n-k+1}i\ell_i=n\\
\sum_{i=1}^{n-k+1}\ell_i=k}}\frac{n!}{\prod_{i=1}^{n-k+1}\ell_i!} \prod_{i=1}^{n-k+1}\biggl(\frac{x_i}{i!}\biggr)^{\ell_i}.
\end{equation*}
The famous Fa\`a di Bruno formula can be described~\cite[p.~139, Theorem~C]{Comtet-Combinatorics-74} in terms of the Bell polynomials of the second kind $\bell_{n,k}(x_1,x_2,\dotsc,x_{n-k+1})$ by
\begin{equation}\label{Bruno-Bell-Polynomial}
\frac{\td^n}{\td x^n}f\circ h(x)=\sum_{k=0}^nf^{(k)}(h(x)) \bell_{n,k}\bigl(h'(x),h''(x),\dotsc,h^{(n-k+1)}(x)\bigr),
\end{equation}
where $f\circ h$ denotes the composite of the $n$-time differentiable functions $f$ and $h$.
\par
Let $h=h(x)=\sqrt{b-x}\,$. Then
\begin{equation*}
h^{(k)}(x)=(-1)^k\biggl\langle\frac{1}{2}\biggr\rangle_k(b-x)^{1/2-k}\to(-1)^k\biggl\langle\frac{1}{2}\biggr\rangle_kb^{1/2-k},\quad x\to0
\end{equation*}
for $k\ge0$ and, in light of the formula~\eqref{Bruno-Bell-Polynomial},
\begin{align*}
\frac{\td^nG_{a,b}(x)}{\td x^n}
&=\sum_{k=0}^{n}\frac{\td^k}{\td h^k}\biggl(\frac{1}{a+h}\biggr)\bell_{n,k}\bigl(h'(x),h''(x),\dotsc,h^{(n-k+1)}(x)\bigr)\\
&=\sum_{k=0}^{n}(-1)^k\frac{k!}{[a+h(x)]^{k+1}}\bell_{n,k}\bigl(h'(x),h''(x),\dotsc,h^{(n-k+1)}(x)\bigr)\\
&\to\sum_{k=0}^{n}(-1)^k\frac{k!}{[a+h(0)]^{k+1}}\bell_{n,k}\biggl(-\biggl\langle\frac{1}{2}\biggr\rangle_1b^{-1/2}, \biggl\langle\frac{1}{2}\biggr\rangle_2b^{-3/2},\\
&\quad\dotsc, (-1)^{n-k+1}\biggl\langle\frac{1}{2}\biggr\rangle_{n-k+1} b^{1/2-(n-k+1)}\biggr), \quad x\to0\\
&=\sum_{k=0}^{n}(-1)^k\frac{k!}{\bigl(a+\sqrt{b}\,\bigr)^{k+1}}(-1)^nb^{k/2-n} \bell_{n,k}\biggl(\biggl\langle\frac{1}{2}\biggr\rangle_1, \biggl\langle\frac{1}{2}\biggr\rangle_2,\dotsc,\biggl\langle\frac{1}{2}\biggr\rangle_{n-k+1}\biggr)\\
&=\frac{1}{2^nb^{n+1/2}}\sum_{k=0}^{n}k! [2(n-k)-1]!!\binom{2n-k-1}{2(n-k)}\biggl(\frac{\sqrt{b}\,}{a+\sqrt{b}\,}\biggr)^{k+1},
\end{align*}
where we used the formulas
\begin{equation*}
\bell_{n,k}\bigl(abx_1,ab^2x_2,\dotsc,ab^{n-k+1}x_{n-k+1}\bigr) =a^kb^n\bell_{n,k}(x_1,x_2,\dotsc,x_{n-k+1})
\end{equation*}
and
\begin{equation}\label{Bell-Polyn-Half}
\bell_{n,k}\biggl(\biggl\langle\frac12\biggr\rangle_1, \biggl\langle\frac12\biggr\rangle_2,\dotsc, \biggl\langle\frac12\biggr\rangle_{n-k+1}\biggr)
=(-1)^{n+k}[2(n-k)-1]!!\biggl(\frac12\biggr)^n\binom{2n-k-1}{2(n-k)}
\end{equation}
for $n\ge k\ge0$, see~\cite[p.~135]{Comtet-Combinatorics-74} and the formula~(3.6) in the first two lines on~\cite[p.~168]{CDM-68111.tex} respectively. By the way, the formula~\eqref{Bell-Polyn-Half} is connected with~\cite[Remark~1]{ODE-Qi-Guo-Kim.tex}, \cite[Section~1.3]{Bell-value-elem-funct.tex}, \cite[Theorem~4]{JAAC961.tex}, and closely related references therein.
\par
The equation~\eqref{C-n(a-b)-gen-funct} means that
\begin{equation*}
n!\mathcal{C}_n(a,b)=\lim_{x\to0}\frac{\td^nG_{a,b}(x)}{\td x^n}.
\end{equation*}
Consequently, we obtain the explicit formula
\begin{equation*}
\mathcal{C}_n(a,b)
=\frac{1}{(2n)!!b^{n+1/2}}\sum_{k=0}^{n}k! [2(n-k)-1]!! \binom{2n-k-1}{2(n-k)} \biggl(\frac{\sqrt{b}\,}{a+\sqrt{b}\,}\biggr)^{k+1},
\end{equation*}
which can be rearranged as~\eqref{C-n(a-b)-explicit}. The proof of Theorem~\ref{C-n(a-b)-explicit-thm} is complete.
\end{proof}

\begin{cor}[{\cite[Theorem~1.3]{JAAC961.tex}}]\label{C-n(a-b)-explicit-cor}
The Catalan number $C_n$ for $n\ge0$ can be explicitly computed by
\begin{equation}\label{C-n-explicit}
C_n=\frac{1}{n!}\sum_{\ell=0}^{n}\binom{n+\ell-1}{2\ell}2^{\ell}(n-\ell)!(2\ell-1)!!.
\end{equation}
\end{cor}

\begin{proof}
This follows from utilizing the relation~\eqref{C-n(ab)=C-n-Eq} and applying $a=\frac12$ and $b=\frac{1}{4}$ in~\eqref{C-n(a-b)-explicit}. The proof of Corollary~\ref{C-n(a-b)-explicit-cor} is complete.
\end{proof}

\begin{rem}
Taking $a\to0^+$ on both sides of the formula~\eqref{C-n(a-b)-explicit} and employing the equation~\eqref{Catlan-Qi-a2zero} result in
\begin{equation*}
\sum_{k=0}^{n}\binom{n+\ell-1}{2\ell}(n-\ell)!(2\ell-1)!!
=(2n-1)!!.
\end{equation*}
Combining~\eqref{C-n-explicit} with the first equality in~\eqref{Catalan-1Exp} gives
\begin{equation*}
\sum_{\ell=0}^{n}\binom{n+\ell-1}{2\ell}(n-\ell)!(2\ell-1)!!2^{\ell}=\frac{n!}{n+1}\binom{2n}{n}.
\end{equation*}
Stimulated by these two identities and the formula~\eqref{C-n(a-b)-explicit} in Theorem~\ref{C-n(a-b)-explicit-thm}, we would like to ask a question: can one use a simple quantity to express the sum
\begin{equation*}
\sum_{\ell=0}^{n}\binom{n+\ell-1}{2\ell}(n-\ell)!(2\ell-1)!!t^{\ell}
\end{equation*}
for $t\in\mathbb{R}\setminus\{0,1,2\}$?
\end{rem}

\section{An integral representation for generalized Catalan numbers of the second kind}

In this section, we will find an integral representation for generalized Catalan numbers of the second kind $\mathcal{C}_n(a,b)$ by finding an integral representation of the generating function $G_{a,b}(x)$ in~\eqref{G(a-b-x)-GF}, which gives a solution to the second problem posed on page~\pageref{C-n(a-b)-gen-funct}.

\begin{thm}\label{G-a-b(x)=int-thm}
The principal branch of the generating function $G_{a,b}(z)$ for $a\ge0$ and $b>0$ can be represented by
\begin{equation}\label{G-a-b(x)=int-Eq}
G_{a,b}(z)=\frac{1}{a+\sqrt{b-z}\,}=\frac{1}{\pi}\int_{0}^{\infty}\frac{\sqrt{t}\,}{a^2+t}\frac1{b+t-z}\td t, \quad z\in\mathbb{C}\setminus[b,\infty).
\end{equation}
Consequently, generalized Catalan numbers of the second kind $\mathcal{C}_n(a,b)$ for $a\ge0$ and $b>0$ can be represented by
\begin{equation}\label{C-n-a-b(x)=int-Eq}
\mathcal{C}_n(a,b)=\frac{1}{\pi}\int_{0}^{\infty}\frac{\sqrt{t}\,}{a^2+t}\frac{1}{(b+t)^{n+1}}\td t, \quad n\ge0.
\end{equation}
\end{thm}

\begin{proof}
Let
\begin{equation*}
F(z)=\frac1{a+\exp\frac{\ln(-z)}2}, \quad z\in\mathbb{C}\setminus[0,\infty), \quad \arg z\in(0,2\pi),
\end{equation*}
where $i=\sqrt{-1}\,$ is the imaginary unit and $\arg z$ stands for the principal value of the argument of $z$.
By virtue of Cauchy's integral formula~\cite[p.~113]{Gamelin-book-2001} in the theory of complex functions, for any fixed point $z_0=x_0+iy_0\in\mathbb{C}\setminus[0,\infty)$, we have
\begin{equation*}
F(z_0)=\frac1{2\pi i}\int_L\frac{F(\xi)}{\xi-z_0}\td\xi,
\end{equation*}
where $L$ is a positively oriented contour $L(r,R)$ in $\mathbb{C}\setminus[0,\infty)$, as showed in Figure~\ref{Catalan-Gen-Int-Formula.jpg},
\begin{figure}[htbp]
  \centering
  \includegraphics[width=0.8\textwidth]{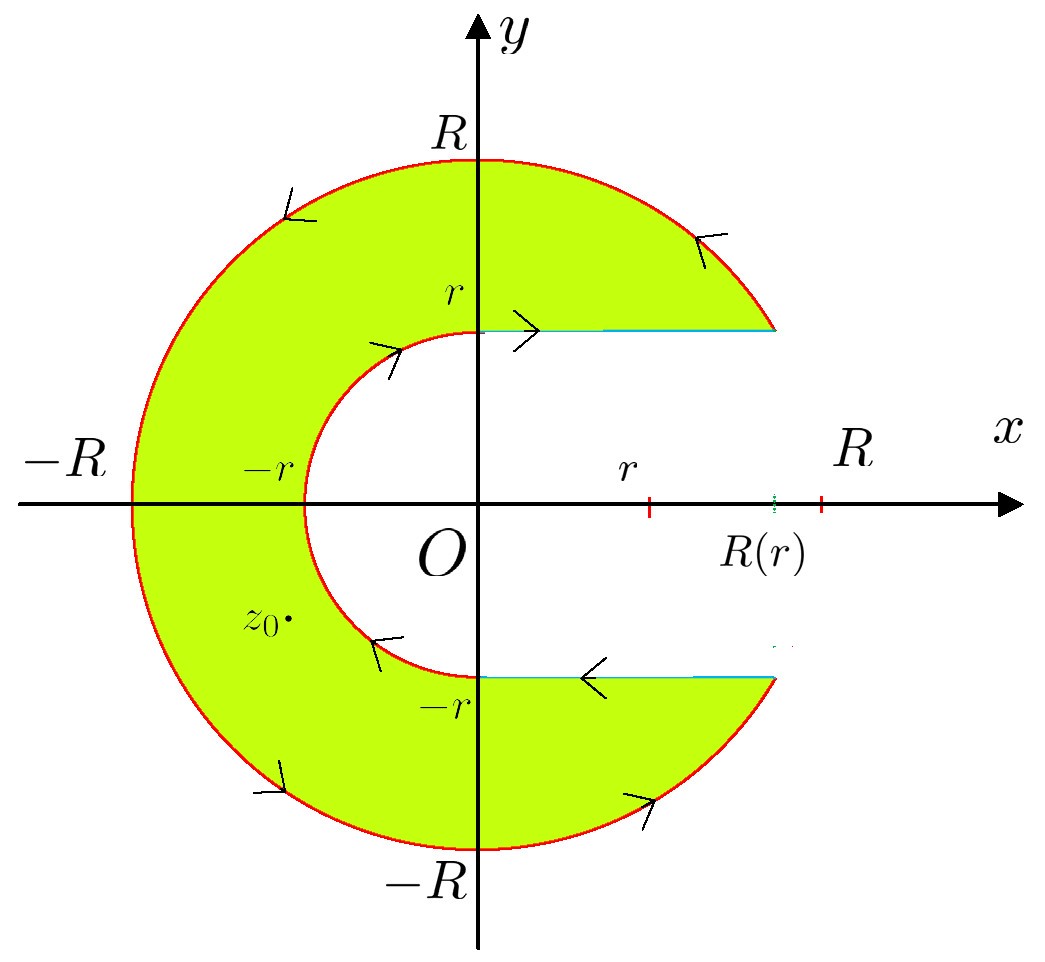}
\caption{The positively oriented contour $L(r,R)$ in $\mathbb{C}\setminus[0,\infty)$}
\label{Catalan-Gen-Int-Formula.jpg}
\end{figure}
satisfying
\begin{enumerate}
\item
$0<r<|z_0|<R$;
\item
$L(r,R)$ consists of the half circle $z=re^{i\theta}$ for $\theta\in\bigl[\frac\pi2,\frac{3\pi}2\bigr]$;
\item
$L(r,R)$ consists of the line segments $z=x\pm ir$ for $x\in(0,R(r)]$, where $R(r)=\sqrt{R^2-r^2}\,$;
\item
$L(r,R)$ consists of the circular arc $z=Re^{i\theta}$ for
$$
\theta\in\biggl(\arctan\frac{r}{R(r)},2\pi-\arctan\frac{r}{R(r)}\biggr);
$$
\item
the line segments $z=x\pm ir$  for $x\in(0,R(r)]$ cut the circle $|z|=R$ at the points $R(r)\pm ir$ and $R(r)\to R$ as $r\to0^+$.
\end{enumerate}
\par
The integral on the circular arc $z=Re^{i\theta}$ equals
\begin{align*}
&\quad\frac1{2\pi i}\int_{\arcsin[r/R(r)]}^{2\pi-\arcsin[r/R(r)]}
\frac{Rie^{i\theta}}{\bigl(Re^{i\theta}-z_0\bigr)\bigl[a+\exp\frac{\ln(-Re^{i\theta})}2\bigr]}\td\theta\\
&=\frac1{2\pi}\int_{\arcsin[r/R(r)]}^{2\pi-\arcsin[r/R(r)]}
\frac1{\bigl(1-\frac{z_0}{Re^{i\theta}}\bigr) \bigl[a+\exp\frac{\ln(-Re^{i\theta})}2\bigr]}\td\theta\\
&=\frac1{2\pi}\int_{\arcsin[r/R(r)]}^{2\pi-\arcsin[r/R(r)]}
\frac1{\bigl(1-\frac{z_0}{Re^{i\theta}}\bigr) \bigl[a+\exp\frac{\ln R+i\arg (-Re^{i\theta})}2\bigr]}\td\theta\\
&\to0
\end{align*}
uniformly as $R\to\infty$.
\par
The integral on the half circle $z=re^{i\theta}$ for $\theta\in\bigl[\frac\pi2,\frac{3\pi}2\bigr]$ is
\begin{align*}
&\quad-\frac1{2\pi i}\int_{\pi/2}^{3\pi/2}
\frac{rie^{i\theta}}{\bigl(re^{i\theta}-z_0\bigr) \bigl[a+\exp\frac{\ln(-re^{i\theta})}2\bigr]}\td\theta\\
&=-\frac1{2\pi}\int_{\pi/2}^{3\pi/2}
\frac{re^{i\theta}}{re^{i\theta}-z_0} \frac1{a+\exp\frac{\ln(-re^{i\theta})}2}\td\theta\\
&=-\frac1{2\pi}\int_{\pi/2}^{3\pi/2}
\frac{re^{i\theta}}{re^{i\theta}-z_0} \frac1{a+\exp\frac{\ln r+i\arg (-re^{i\theta})}2}\td\theta\\
&=-\frac1{2\pi}\int_{\pi/2}^{3\pi/2}
\frac{re^{i\theta}}{re^{i\theta}-z_0} \frac1{a+\sqrt{r}\,\exp\frac{i\arg (-re^{i\theta})}2}\td\theta\\
&\to0
\end{align*}
uniformly as $r\to0^+$.
\par
Since
\begin{align*}
F(x+ir)&=\frac1{a+\exp\frac{\ln(-x-ri)}2}\\
&=\frac1{a+\exp\frac{\ln\sqrt{x^2+r^2}\,+i[\arctan(r/x)-\pi]}2}\\
&=\frac1{a+\sqrt[4]{x^2+r^2}\, \bigl[\cos\frac{\arctan(r/x)-\pi}2 +i\sin\frac{\arctan(r/x)-\pi}2\bigr]}\\
&=\frac1{a+\sqrt[4]{x^2+r^2}\, \bigl[\sin\frac{\arctan(r/x)}2 -i\cos\frac{\arctan(r/x)}2\bigr]}\\
&=\frac1{a+\sqrt[4]{x^2+r^2}\,\sin\frac{\arctan(r/x)}2 -i\sqrt[4]{x^2+r^2}\,\cos\frac{\arctan(r/x)}2}\\
&=\frac{a+\sqrt[4]{x^2+r^2}\,\sin\frac{\arctan(r/x)}2
+i\sqrt[4]{x^2+r^2}\,\cos\frac{\arctan(r/x)}2}
{\bigl[a+\sqrt[4]{x^2+r^2}\,\sin\frac{\arctan(r/x)}2\bigr]^2
+\bigl[\sqrt[4]{x^2+r^2}\,\cos\frac{\arctan(r/x)}2\bigr]^2}\\
&\to\frac{a+i\sqrt{x}\,}{a^2+x}
\end{align*}
as $r\to0^+$ and $\overline{F(z)}=F(\bar{z})$, the integral on the line segments $z=x\pm ir$  for $x\in(0,R(r)]$ is equal to
\begin{align*}
&\quad\frac1{2\pi i}\Biggl[\int_0^{R(r)}\frac {F(x+ir)}{x+ir-z_0}\td x
+\int_{R(r)}^0\frac {F(x-ir)}{x-ir-z_0}\td x\Biggr]\\
&=\frac1{2\pi i}\int_0^{R(r)}\frac{(x-ir-z_0)F(x+ir)-(x+ir-z_0)F(x-ir)}{(x+ir-z_0)(x-ir-z_0)}\td x\\
&=\frac1{2\pi i}\int_0^{R(r)} \frac{(x-z_0)[F(x+ir)-F(x-ir)]-ir[F(x+ir)+F(x-ir)]}{(x+ir-z_0)(x-ir-z_0)}\td x\\
&=\frac1{2\pi i}\int_0^{R(r)}
\frac{(x-z_0)\bigl[F(x+ir)-F\bigl(\overline{x+ir}\bigr)\bigr] -ir\bigl[F(x+ir)+F\bigl(\overline{x+ir}\bigr)\bigr]}
{(x+ir-z_0)(x-ir-z_0)}\td x\\
&=\frac1{2\pi i}\int_0^{R(r)}\frac{(x-z_0)\bigl[F(x+ir)-\overline{F(x+ir)}\,\bigr] -ir\bigl[F(x+ir)+\overline{F(x+ir)}\,\bigr]}{(x+ir-z_0)(x-ir-z_0)}\td x\\
&=\frac1{2\pi i}\int_0^{R(r)}\frac{(x-z_0)[2i\Im(F(x+ir))] -ir[2\Re(F(x+ir))]}{(x+ir-z_0)(x-ir-z_0)}\td x\\
&\to\frac1{2\pi i}\int_0^\infty\frac{2i}{x-z_0}\frac{\sqrt{x}\,}{a^2+x}\td x, \quad r\to0^+, \quad R\to\infty\\
&=\frac1\pi\int_0^\infty\frac{\sqrt{x}}{(a^2+x)(x-z_0)}\td x.
\end{align*}
Consequently, it follows that
\begin{equation}\label{integral-form-F(z0)}
\frac1{a+\exp\frac{\ln(-z_0)}2}=\frac1{\pi}\int_0^{\infty}\frac{\sqrt{x}}{(a^2+x)(x-z_0)}\td x
\end{equation}
for any $z_0\in\mathbb{C}\setminus[0,\infty)$ and $\arg z_0\in(0,2\pi)$. Due to the point $z_0$ in~\eqref{integral-form-F(z0)} being arbitrary, the integral formula~\eqref{integral-form-F(z0)} can be rearranged as
\begin{equation}\label{integral-form-F(z)}
F(z)=\frac1{a+\exp\frac{\ln(-z)}2}=\frac1{\pi}\int_0^{\infty}\frac{\sqrt{x}}{(a^2+x)(x-z)}\td x
\end{equation}
for $z\in\mathbb{C}\setminus[0,\infty)$ and $\arg z\in(0,2\pi)$.
\par
Let
\begin{equation*}
f(z)=\frac1{a+\exp\frac{\ln(b-z)}2}, \quad z\in\mathbb{C}\setminus[b,\infty), \quad \arg(z-b)\in(0,2\pi).
\end{equation*}
Then $f(z)=F(z-b)$. Therefore, from~\eqref{integral-form-F(z)}, it follows that
\begin{equation*}
f(z)=\frac1{a+\exp\frac{\ln(b-z)}2}=\frac1{\pi}\int_0^{\infty}\frac{\sqrt{x}}{(a^2+x)(b+x-z)}\td x
\end{equation*}
for $z\in\mathbb{C}\setminus[b,\infty)$ and $\arg(z-b)\in(0,2\pi)$. The integral representation~\eqref{G-a-b(x)=int-Eq} is thus proved.
\par
Differentiating $n\ge0$ times with respect to $z$ on both sides of~\eqref{G-a-b(x)=int-Eq} and taking the limit $z\to0$ yield
\begin{align*}
G_{a,b}^{(n)}(z)&=\frac{\td^n}{\td z^n}\biggl(\frac{1}{a+\sqrt{b-z}\,}\biggr)\\
&=\frac{1}{\pi}\int_{0}^{\infty}\frac{\sqrt{t}\,}{a^2+t}\frac{\td^n}{\td z^n}\biggl(\frac1{b+t-z}\biggr)\td t\\
&=\frac{1}{\pi}\int_{0}^{\infty}\frac{\sqrt{t}\,}{a^2+t}\frac{n!}{(b+t-z)^{n+1}}\td t\\
&\to\frac{n!}{\pi}\int_{0}^{\infty}\frac{\sqrt{t}\,}{a^2+t}\frac{1}{(b+t)^{n+1}}\td t, \quad z\to0.
\end{align*}
As a result, by virtue of~\eqref{C-n(a-b)-gen-funct}, we have
\begin{equation*}
\mathcal{C}_n(a,b)=\frac{G_{a,b}^{(n)}(0)}{n!}=\frac{1}{\pi}\int_{0}^{\infty}\frac{\sqrt{t}\,}{a^2+t}\frac{1}{(b+t)^{n+1}}\td t.
\end{equation*}
The integral representation~\eqref{C-n-a-b(x)=int-Eq} for generalized Catalan numbers of the second kind $\mathcal{C}_n(a,b)$ is thus proved.
The proof of Theorem~\ref{G-a-b(x)=int-thm} is complete.
\end{proof}

\begin{rem}
When taking $z=x\in(-\infty,b)$, the integral formula~\eqref{G-a-b(x)=int-Eq} becomes
\begin{equation}\label{G-a-b(z2x)=int-Eq}
\frac{1}{a+\sqrt{b-x}\,}=\frac{1}{\pi}\int_{0}^{\infty}\frac{\sqrt{t}\,}{a^2+t}\frac1{b+t-x}\td t, \quad a\ge0,b>0.
\end{equation}
When taking $x\to b^-$, the integral in~\eqref{G-a-b(z2x)=int-Eq} converges. Consequently, the integral representation~\eqref{G-a-b(z2x)=int-Eq} is valid on $(-\infty,b]$.
\end{rem}

\begin{rem}
When taking $a=\frac{1}{2}$ and $b=\frac{1}{4}$, integral representations~\eqref{G-a-b(x)=int-Eq} and~\eqref{C-n-a-b(x)=int-Eq} reduce to~\eqref{real-integral-form} and~\eqref{Catalan-No-Int-Qi} respectively.
\end{rem}

\begin{rem}
Combining the explicit formula~\eqref{C-n(a-b)-explicit} with the integral representation~\eqref{C-n-a-b(x)=int-Eq} and simplifying lead to
\begin{equation*}
\int_{0}^{\infty}\frac{\sqrt{t}\,}{a^2+t}\frac{1}{(b+t)^{n+1}}\td t
=\frac{\pi}{(2n)!!b^{n+1/2}}\sum_{k=0}^{n}\binom{2n-k-1}{2(n-k)}\frac{k![2(n-k)-1]!!}{\bigl(1+a/\sqrt{b}\,\bigr)^{k+1}}
\end{equation*}
for $a\ge0$, $b>0$, and $n\ge0$.
\end{rem}

\section{Potential directions to further study}

In this section, we will try to point out two potential directions to further study.

\subsection{Generalized Catalan function of the second kind}
Motivated by the integral representation~\eqref{C-n-a-b(x)=int-Eq} for generalized Catalan numbers of the second kind $\mathcal{C}_n(a,b)$, we can consider the function
\begin{equation}\label{C-n-a-b(x)=CQ-Eq}
\mathcal{C}(a,b;z)=\frac{1}{\pi}\int_{0}^{\infty}\frac{\sqrt{t}\,}{a^2+t}\frac{1}{(b+t)^{z+1}}\td t, \quad a\ge0,b>0,\Re(z)\ge0
\end{equation}
and call it generalized Catalan function of the second kind, while calling $C(a,b;z)$ in~\eqref{C(a,b;x)=dfn} generalized Catalan function of the first kind.
\par
We can study generalized Catalan function of the second kind $\mathcal{C}(a,b;z)$ as a function of three variables $a,b$, and $z$. It is easy to see that
\begin{align*}
\frac{\td^n\mathcal{C}(a,b;z)}{\td b^n}&=\frac{1}{\pi}\int_{0}^{\infty}\frac{\sqrt{t}\,}{a^2+t} \frac{\td^n}{\td b^n} \biggl[\frac{1}{(b+t)^{z+1}}\biggr]\td t\\
&=\frac{1}{\pi}\int_{0}^{\infty}\frac{\sqrt{t}\,}{a^2+t}\frac{\langle-(z+1)\rangle_n}{(b+t)^{z+n+1}}\td t\\
&=(-1)^n\frac{(z+1)_n}{\pi}\int_{0}^{\infty}\frac{\sqrt{t}\,}{a^2+t}\frac{1}{(b+t)^{z+n+1}}\td t,
\end{align*}
where the rising factorial $(z)_n$ is defined~\cite{CDM-68111.tex, Note-Elem-UPB-Qi-Zhang-Li.tex} by
\begin{equation*}
(z)_n=\prod_{\ell=0}^{n-1}(z+\ell)
=
\begin{cases}
z(z+1)\dotsm(z+n-1), & n\ge1\\
1, & n=0
\end{cases}
\end{equation*}
which is also known as the Pochhammer symbol or shifted factorial in the theory of special functions~\cite{NIST-HB-2010, Temme-96-book}. This means that generalized Catalan function of the second kind $\mathcal{C}(a,b;z)$ is a completely monotonic function~\cite{Qi-Agar-Surv-JIA.tex, Schilling-Song-Vondracek-2nd, widder} with respect to $b\in(0,\infty)$. Utilizing complete monotonicity~\cite{AIMS-Math-2019595.tex, Schilling-Song-Vondracek-2nd, widder}, we can derive many new analytic properties of generalized Catalan function of the second kind $\mathcal{C}(a,b;z)$.
\par
In a word, employing the integral representation~\eqref{C-n-a-b(x)=CQ-Eq}, we believe that we can discover some new properties of generalized Catalan function of the second kind $\mathcal{C}(a,b;z)$, of generalized Catalan numbers of the second kind $\mathcal{C}_n(a,b)$, and the Catalan numbers $C_n$. For the sake of length limit of this paper, we would not like to further study in details.

\subsection{Central binomial coefficients}

It is known that
\begin{equation}\label{central-binomial-GF}
\frac{1}{\sqrt{1-4x}\,}=\sum_{n=0}^{\infty}\binom{2n}{n}x^n
=1+ 2x+ 6x^2+ 20x^3+70x^4+ 252x^5+\dotsm,
\end{equation}
where $\binom{2n}{n}$ is called central binomial coefficient. It has been an attracting point for mathematicians to study central binomial coefficients. For example, we can rewritten~\cite[Lemma~3]{elliptic-mean-comparison-rev2.tex} as
\begin{equation*}
\sum_{k=0}^{n-1}\binom{2k}{k}\frac{1}{(k+1)4^{k}}=2\biggl[1-\frac{1}{4^{n}}\binom{2n}{n}\biggr]
\end{equation*}
and
\begin{equation*}
\sum_{k=0}^{n-1}\binom{2k}{k}\frac{4^{n-k}}{n-k}=2\binom{2n}{n}\sum_{k=1}^{n}\frac{1}{2k-1}.
\end{equation*}
For more information on results at this point, please refer to~\cite{Boyad-JIS-2012, John-Maxwell-Rocky-2019, Chen-HW-IJS-2016, Gar, Guo-Wang-Edinb-2020, Lehmer-Monthly-1985, Jovan-Mikic-2020, Qi-Chen-Lim-RNA.tex, Catalan-Int-Surv.tex, elliptic-mean-comparison-rev2.tex, Sprugnoli-INT-2006, Yin-Qi-20161028.tex} and closely related references therein.
\par
Combining~\eqref{central-binomial-GF} with~\eqref{C-n(a-b)-gen-funct} and~\eqref{C-n-a-b(x)=int-Eq} arrives at
\begin{align*}
\binom{2n}{n}&=\frac{1}{n!}\lim_{x\to0}\frac{\td^n}{\td x^n}\biggl(\frac{1}{\sqrt{1-4x}\,}\biggr)\\
&=\frac{1}{n!}\lim_{x\to0}\frac{\td^n}{\td x^n}\left(\frac12\lim_{\substack{a\to0^+\\ b\to1/4}}\frac1{a+\sqrt{b-x}\,}\right)\\
&=\frac{1}{n!}\frac12\lim_{\substack{a\to0^+\\ b\to1/4}}\lim_{x\to0}\frac{\td^n}{\td x^n}\biggl(\frac1{a+\sqrt{b-x}\,}\biggr)\\
&=\frac12\lim_{\substack{a\to0^+\\ b\to1/4}}\mathcal{C}_n(a,b)\\
&=\frac12\lim_{\substack{a\to0^+\\ b\to1/4}}\frac{1}{\pi}\int_{0}^{\infty}\frac{\sqrt{t}\,}{a^2+t}\frac{1}{(b+t)^{n+1}}\td t\\
&=\frac{1}{2\pi}\int_{0}^{\infty}\frac1{\sqrt{t}\,}\frac{1}{(1/4+t)^{n+1}}\td t\\
&=\frac{1}{\pi}\int_{0}^{\infty}\frac{1}{(1/4+s^2)^{n+1}}\td s\\
&=\frac{2^{2n+1}}{\pi}\int_{0}^{\infty}\frac{1}{(1+t^2)^{n+1}}\td t.
\end{align*}
The last three integral representations should provide effective tools for further studying central binomial coefficients $\binom{2n}{n}$.
\par
There are several extensions of central binomial coefficients $\binom{2n}{n}$ in the paper~\cite{DIGCBC-Wei.tex}.

\begin{rem}
This paper is a revised version of the preprint~\cite{Catalan-Gen-Int-Formula.tex} and a companion of the electronic preprint~\cite{Li-Qi-Authorea-2020, Catalan-Li-Qi-Kouba.tex}.
\end{rem}

\section{Declarations}

\begin{description}

\item[Acknowledgements]
The authors thank anonymous referees for their careful corrections to, valuable comments on, and helpful suggestions to the original version of this paper.

\item[Funding]
Not applicable.

\item[Availability of data and material]
Data sharing is not applicable to this article as no new data were created or analyzed in this study.

\item[Competing interests]
The authors declare that they have no conflict of competing interests.

\item[Authors' contributions]
All authors contributed equally to the manuscript and read and approved the final manuscript.
\end{description}

\end{document}